\documentclass[11pt]{amsart}
\usepackage[margin=1in]{geometry}
\usepackage{amsfonts, amstext, amsmath, amssymb}
\usepackage{graphicx}

\newtheorem{lemma}{Lemma}
\newtheorem{theorem}[lemma]{Theorem}
\newtheorem{proposition}[lemma]{Proposition}
\newtheorem{corollary}[lemma]{Corollary}

\begin{document}

\title{Constructing the demand function of a strictly convex preference relation}

\author[Hendtlass]{Matthew Hendtlass}
 \address{School of Mathematics and Statistics,
    University of Canterbury,
    Christchurch 8041,
    New Zealand}
\email{matthew.hendtlass@canterbury.ac.nz}

\begin{abstract}
\noindent
We give conditions under which the demand function of a strictly convex preference relation can be constructed.
\end{abstract}

\maketitle

\section*{Introduction}

\noindent
This paper gives conditions under which the demand function of a strictly convex preference relation can be constructed, and should be seen as a continuation of the work of Douglas Bridges \cite{dsb_82,dsb_89,dsb_ctsdem,dsb_94} to examine aspects of mathematical economics in a rigorously constructive manner, see also \cite{HM}. In particular, Bridges considered the problem that we consider here in \cite{dsb_ctsdem}. Corollary \ref{DF2} is a generalisation of the main result of \cite{dsb_ctsdem} and our proof, although less elegant, is also somewhat simpler. 

\bigskip
\noindent
Following Bridges we take, as our starting point, the standard configuration in microeconomics consisting of a consumer whose consumption set $X$ is a compact, convex subset of $\mathbf{R}^n$ ordered by a strictly ordered preference relation $\succ$. For a given price vector $p\in\mathbf{R}^n$ and a given initial endowment $w$, the consumers \emph{budget set}
 $$
  \beta(p,w)=\{x\in X:p\cdot x\leqslant w\}
 $$
\noindent
is the collection of all consumption bundles available to the consumer. 

\bigskip
\noindent
As detailed in \cite{dsb_ctsdem}, it is easy to show that classically, if $\beta(p,w)\neq\emptyset$, then there exists a unique $\succ$-\emph{maximal} point $\xi_{p,w}\in\beta(p,w)$: $\xi_{p,w}\succcurlyeq x$ for all $x\in\beta(p,w)$. Let $T$ be the set of pairs consisting of a price vector $p$ and an initial endowment $w$ for which $\beta(p,w)$ is inhabited. If the preference relation $\succ$ is continuous, then a sequential compactness argument gives the sequential, and hence pointwise, continuity of the \emph{demand function} $F$ on $T$ which sends $(p,w)$ to the maximal element $\xi_{p,w}$ of $\beta(p,w)$ (see, for example, chapter 2, section D of \cite{Takayama}). 

\bigskip
\noindent
Bridges asked under what conditions can we
 \begin{itemize}
  \item[1.] Compute the demand function $F$;
  \item[2.] Compute a modulus of uniform continuity for $F$: given $\varepsilon>0$, can we produce $\delta>0$ such that if $(p,w),(p^\prime,w^\prime)\in T$ with $\Vert(p,w)-(p^\prime,w^\prime)\Vert<\delta$, then $\Vert F(p,w)-F(p^\prime,w^\prime)\Vert<\varepsilon$.
 \end{itemize}
\noindent
In \cite{dsb_ctsdem} Bridges introduced the notion of a uniformly rotund preference relation and showed that if $\succ$ is uniformly rotund and you restrict $F$ to a compact subset of $T$ on which the consumer cannot be satiated, then $F$ is uniformly continuous. Theorem \ref{DF2} shows that we do not need the hypothesis that our consumer is nonsatiated. Theorems \ref{t1} and \ref{DF1} encapsulate what we can say about strictly convex preference relations, which is more than one might think.

\bigskip
\noindent
We work in Bishop's style constructive mathematics. Any proof in this framework embodies an algorithm, so when we show that there exists $x$ such that $P(x)$, our proof gives an explicit construction of an object $x$ together with a proof that $P(x)$ holds. Formally we take Bishop's constructive mathematics to be Aczel's constructive Zermelo-Fraenkel set theory (\textbf{CZF}) with intuitionistic logic and the axiom of dependent choice \cite{AR}. By interpreting \textbf{CZF} in Martin-L\"{o}f type theory \cite{A_CZF}, the algorithmic nature of our proofs can be made explicit; and techniques for programme extraction from such proofs have been well studied \cite{HS}. We direct the reader to \cite{BB,BV} for an introduction to the practice of Bishop's constructive mathematics, and to \cite{dsb_82,dsb_89} for an introduction to Bridges' programme to constructivise mathematical economics.

\bigskip
\noindent
A \emph{preference relation} $\succ$ on a set $X$ is a binary relation which is
 \begin{itemize}
  \item[$\blacktriangleright$] asymmetric: if $x\succ y$, then $\neg(y\succ x)$;
  \item[$\blacktriangleright$] negatively transitive: if $x\succ y$, then for all $z$ either $x\succ z$ or $z\succ y$.
 \end{itemize}
\noindent
If $x\succ y$, we say that $x$ is \emph{preferable} to $y$. We write $x\succcurlyeq y$, $x$ is \emph{preferable or indifferent to} $y$, for $\neg(y\succ x)$. We note that $x\succ x$ is contradictory, that $\succ$ and $\succcurlyeq$ are transitive, and that if either $x\succcurlyeq y\succ z$ or $x\succ y\succcurlyeq z$, then $x\succ z$.

\bigskip
\noindent
Let $\succ$ be a preference relation on a subset $X$ of $\mathbf{R}^N$.
 \begin{itemize}
  \item[$\blacktriangleright$] $\succ$ is a \emph{continuous preference relation} if the graph 
   $$
    \{(x,x^\prime):x\succ x^\prime\}
   $$
  \noindent
  of $\succ$ is open.
 \item[$\blacktriangleright$] $\succ$ is \emph{strictly convex} if $X$ is convex and $tx+(1-t)x^\prime\succ x$ or $tx+(1-t)x^\prime\succ x^\prime$ whenever $x,x^\prime\in X$ with $x\neq x^\prime$ and $t\in(0,1)$.
 
 \item[$\blacktriangleright$] $X$ is \emph{uniformly rotund} if for each $\varepsilon>0$ there exists $\delta>0$ such that for all $x,x^\prime\in X$, if $\Vert x-x^\prime\Vert\geqslant\varepsilon$, then 
 $$
  \left\{\frac{1}{2}\left(x+x^\prime\right)+z:z\in B(0,\delta)\right\}\subset X,
 $$
\noindent
where $B(x,r)$ is the open ball of radius $r$ centred on $x$. The preference relation $\succ$ is uniformly rotund if $X$ is uniformly rotund and for each $\varepsilon>0$ there exists $\delta>0$ such that if $\Vert x-x^\prime\Vert\geqslant\varepsilon$ ($x,x^\prime\in X$), then for each $z\in B(0,\delta)$ either $\frac{1}{2}\left(x+x^\prime\right)+z\succ x$ or $\frac{1}{2}\left(x+x^\prime\right)+z\succ x^\prime$.
 \end{itemize}
\noindent
A uniformly rotund preference relation is strictly convex. 

\bigskip
\noindent
A set $S$ is said to be \emph{inhabited} if there exists $x$ such that $x\in S$. An inhabited subset $S$ of a metric space $X$ is \emph{located} if for each $x\in X$ the \emph{distance} 
\begin{equation*}
\rho \left( x,S\right) =\inf \left\{ \rho (x,s):s\in S\right\}
\end{equation*}%
\noindent 
from $x$ to $S$ exists. If $X$ is located and its \emph{metric complement} 
 $$
  -X=\{x\in\mathbf{R}^n:\rho(x,X)>0\}
 $$
\noindent
is also located, then $X$ is said to be \emph{bilocated}. An $\varepsilon $%
\emph{-approximation} to $S$ is a subset $T$ of $S$ such
that for each $s\in S$, there exists $t\in T$ such that $\rho
(s,t)<\varepsilon $. We say that $S$ is \emph{totally bounded} if for each $%
\varepsilon >0$ there exists a finitely enumerable\footnote{A set is \emph{finitely enumerable} if it is the image of $\{1,\ldots,n\}$ for some $n\in N$, and a set is \emph{finite} if it is in bijection with $\{1,\ldots,n\}$ for some $n\in N$; constructively these notions are distinct.} $\varepsilon $-approximation to $S$; totally bounded sets are located. A metric space $X$ is \emph{compact} if it is complete and totally bounded. We will use $\Vert\cdot\Vert$ to represent the Euclidean norm, $\Vert\cdot\Vert_1$ for the norm $x\mapsto\sum_{i=1}^nx_i$ on $\mathbf{R}^n$, and we write $\rho,\rho_1$ for the respective induced metrics.

\section*{Constructing maxima}

\noindent
In this section we focus on the construction of maximally prefered elements of a consumption set $X$. Our main result is

\begin{theorem}
\label{t1}
Let $\succ$ be a continuous, strictly convex preference relation on an inhabited, compact subset $X$ of Euclidean space. Then there exists a unique $\xi\in X$ such that $\xi\succcurlyeq x$ for all $x\in X$.
\end{theorem}

\noindent
Our proof proceeds by induction. The following lemma provides the key to proving the one dimensional case.

\begin{lemma}
\label{l1}
Let $\succ$ be a strictly convex preference relation on $[0,1]$. Then either $1/2\succcurlyeq x$ for all $x\in[0,1/4)$ or $1/2\succcurlyeq x$ for all $x\in(3/4,1]$.
\end{lemma}

\begin{proof}
Applying the strict convexity of $\succ$ to $1/4\in(0,3/4),1/2\in(1/4,3/4),3/4\in(1/2,1)$ yields
 \begin{eqnarray*}
  1/4\succ 0 & \mathrm{or} & 1/4\succ 3/4;\\
  1/2\succ 1/4 & \mathrm{or} & 1/2\succ 3/4;\\
  3/4\succ 1/4 & \mathrm{or} & 3/4\succ 1.
 \end{eqnarray*}
\noindent
It follows that either $1/2\succ1/4\succ0$ or $1/2\succ3/4\succ1$. In the first case suppose that $x\succcurlyeq1/2$ for some $x\in[0,1/4)$. Then, by the strict convexity and transitivity of $\succ$,  $1/4\succ 1/2$; this contradiction ensures that $1/2\succcurlyeq x$ for all $x\in[0,1/4)$. Similarly, in the second case $1/2\succcurlyeq x$ for all $x\in(3/4,1]$.
\end{proof}

\begin{lemma}
\label{lb}
 If $\succ$ is a strictly convex, continuous preference relation on $[0,1]$, then there exists $\xi\in[0,1]$ such that $\xi\succcurlyeq x$ for all $x\in [0,1]$.
\end{lemma}

\begin{proof}
We inductively construct intervals $[\underline{\xi}_n,\overline{\xi}_n]$ such that, for each $n$,
  \begin{itemize}
   \item[1.] $[\underline{\xi}_{n},\overline{\xi}_{n}]\subset[\underline{\xi}_{n-1},\overline{\xi}_{n-1}]$;
   \item[2.] $\overline{\xi}_{n}-\underline{\xi}_{n}=(4/5)^n$;
   \item[3.] for each $x\in [0,1]\setminus[\underline{\xi}_n,\overline{\xi}_n]$,there exists $y\in[\underline{\xi}_n,\overline{\xi}_n]$ such that $y\succcurlyeq x$.
  \end{itemize}
\noindent
To begin the construction set $\underline{\xi}_0=0$ and $\overline{\xi}_0=1$. At stage $n$, rescaling for $n>1$, we apply Lemma \ref{l1}; if the first case obtains, then we set $\underline{\xi}_n=(3\underline{\xi}_{n-1}+\overline{\xi}_{n-1})/4$ and $\overline{\xi}_n=\overline{\xi}_{n-1}$. In the second case we set $\underline{\xi}_n=\underline{\xi}_{n-1}$ and $\overline{\xi}_0=(\underline{\xi}_{n-1}+3\overline{\xi}_{n-1})/4$. By the transitivity of $\succcurlyeq$, we need only check condition 3. for $[\underline{\xi}_{n-1},\overline{\xi}_{n-1}]\setminus[\underline{\xi}_n,\overline{\xi}_n]$, and by Lemma \ref{l1} $y=(\underline{\xi}_{n-1}+\overline{\xi}_{n-1})/2$ suffices for each such point.

\bigskip
\noindent
Let $\xi$ be the unique intersection of the $[\underline{\xi}_n,\overline{\xi}_n]$. Since $\succcurlyeq$ is continuous, the maximality of $\xi$ follows from 3. 
\end{proof}

\begin{lemma}
\label{l3}
 If $\succ$ is a strictly convex, continuous preference relation on $[a,b]$, where $a\leqslant b$, then there exists $\xi\in[a,b]$ such that $\xi\succcurlyeq x$ for all $x\in [a,b]$.
\end{lemma}

\begin{proof}
Construct an increasing binary sequence $(\lambda_n)_{n\geqslant1}$ such that
 \begin{eqnarray*}
  \lambda_n=0 & \Rightarrow & b-a<1/n;\\
  \lambda_n=1 & \Rightarrow & b-a>1/(n+1).
 \end{eqnarray*}
\noindent
Without loss of generality, we may assume that $\lambda_1=0$. If $\lambda_n=0$, set $x_n=a$ and if $\lambda_n=1-\lambda_{n-1}$, then we apply Lemma \ref{lb}, after some scaling, to construct a $\succ$-maximal element $x$ in $[a,b]$, and set $x_k=x$ for all $k\geqslant n$. Then for $m>n$, $\vert x_n-x_m\vert<2/(n-1)$, so $(x_n)_{n\geqslant1}$ converges to some element $\xi\in[a,b]$. If there exists $x\neq\xi$ such that $x\succ\xi$, then $b-a>0$ and we get a contradiction to Lemma \ref{lb}. The result now follows from continuity.
\end{proof}

\bigskip
\noindent
We use $\pi_i$ to denote the $i$-th projection function, and we write $[x,y]$ for 
 \[
  \{tx+(1-t)y:t\in[0,1]\}.
 \]
\noindent
Here is the \textbf{proof of Theorem \ref{t1}:}

\bigskip
\begin{proof}
We proceed by induction on the dimension $n$ of the space containing $X$. Lemma \ref{l3} is just the case $n=1$. Now suppose we have proved the result for $n$ and consider a strictly convex preference relation $\succ$ on a compact, convex subset $X$ of $\mathbf{R}^n$. Define a preference relation $\succ^\prime$ on $\pi_1(X)=[a,b]$ by
 $$
   s\succ_i t\ \ \ \Leftrightarrow\ \ \ \exists_{x\in X}\forall_{y\in X}\left( \pi_1(x)=s\mbox{ and if } \pi_1(y)=t\mbox{, then } x\succ y\right).
 $$
\noindent
Then $\succ^\prime$ is strictly convex and sequentially continuous: let $s_1,s_2,t\in[a,b]$ with $s_1<t<s_2$. By the induction hypothesis there exist $\xi_1,\xi_2$ such that $\pi_1(\xi_i)=s_i$ and $\xi_i\succcurlyeq x$ for all $x\in X$ with $\pi_1(x)=s_i$ ($i=1,2$). Let $z$ be the unique element of $[\xi_1,\xi_2]$ such that $\pi_1(z)=t$. Then, by the strict convexity of $\succ$, either $z\succ\xi_1$ or $z\succ\xi_2$. In the first case $t\succ^\prime s_1$ and in the second $t\succ^\prime s_2$. Hence $\succ$ is strictly convex. That $\succ^\prime$ is continuous is straightforward.

\bigskip
\noindent
We can now apply Lemma \ref{l3} to construct a maximal element $\xi_1$ of $(\pi_1(X),\succ^\prime)$, and then the induction hypothesis to construct a maximal element of 
 $$
  S=\{x\in X:\pi_1(x)=\xi_1\}
 $$
\noindent
with $\succ|_{S}$. Clearly $\xi=\xi_1\times\xi_2$ is a $\succ$-maximal element of $X$. The uniqueness of maximal elements follows directly from the strict convexity of $\succ$.
\end{proof}

\bigskip
\noindent
We shall have need for the following simple corollary, which is of independent interest.

\begin{corollary}
\label{c1}
Under the conditions of Theorem \ref{t1}, if $x\in X$ and $x\neq\xi$, then $\xi\succ x$.
\end{corollary}

\begin{proof}
 Let $y=(x+\xi)/2$. Then either $y\succ x$ or $y\succ\xi$. Since $\xi\succcurlyeq y$ the former must attain, so $\xi\succcurlyeq y\succ x$.
\end{proof}

\bigskip
\noindent
If we are not interested in uniqueness of maxima, then we might suppose that $\succ$ only satisfies the weaker condition of being \emph{convex}: for all $x,y\in X$ and each $t\in[0,1]$, either $(x+y)/2\succcurlyeq x$ or $(x+y)/2\succcurlyeq y$. We give a Browuerian counterexample\footnote{A Brouwerian counterexample is a weak counterexample: it is not an example contradicting a proposition, but an example showing a proposition to imply a principle which is unacceptable in constructive mathematics. Generally these can be considered as unprovability results.} to show that this condition is not strong enough to allow the construction of a maximal point. Let $x\in(-1/4,1/4)$ and let $f:[0,1]\rightarrow\mathbf{R}$ be the function given by
 $$
  f(t)=\left\{\begin{array}{lll}
                \mathrm{sign}(x)(t-x\vee0) & & t\in[0,x\vee0]\\
                0 & & t\in[x\vee0,1-x\vee0]\\
                -\mathrm{sign}(x)(t-x\vee0) & & t\in[1-x\vee0,1],
               \end{array}\right.
 $$
\noindent
where\footnote{This is just convenient notation: formally $\mathrm{sign}$ is not a constructively well defined function, but the function $f$ does exist constructively.}
 $$
  \mathrm{sign}(x)=\left\{\begin{array}{lll}
                -1 & & x<0\\
                0 & & x=0\\
                1 & & x>0.
               \end{array}\right.
 $$
\noindent
Define a preference realtion $\succ$ on $[0,1]$ by 
 $$
  t\succ s \Leftrightarrow f(t)>f(s).
 $$
\noindent
It is easy to see that $\succ$ is continuous and convex. Further, if $x>0$, then $0$ is the unique maximal element, and if $x<0$, then $1$ is the unique maximal element. Now suppose that we can construct $\xi\in[0,1]$ such that $\xi\succcurlyeq t$ for all $t\in[0,1]$; either $\xi>0$ or $\xi<1$. In the first case we have $\neg(x>0)$ and in the second $\neg(x<0)$, so the statement
 \begin{quote}
  `Every continuous, convex preference relation on $[0,1]$ has a maximal element'
 \end{quote}
\noindent
implies $\forall_{x\in\mathbf{R}}(x\leqslant0\vee x\geqslant0)$, which is equivalent 
to the constructively unacceptable lesser limite principle of omniscience \cite{BR}.

\section*{Continuous demand functions}

\noindent
We now consider a consumer whose consumption set $X$ is a closed convex subset of $\mathbf{R}^n$
ordered by a strictly convex preference relation $\succ$, and who has an initial
endowment $w\in\mathbf{R}$. For a given price vector $p\in\mathbf{R}$, a consumers \emph{budget set}
 $$
  \beta(p,w)=\{x\in X:p.x\leqslant w\}
 $$
\noindent
is the collection of commodity bundles the consumer can afford. The collection of maximal elements of $\beta(p,w)$ form the consumers \emph{demand set} for price $p$ and initial endowment $w$.

\begin{lemma}
\label{dsb}
 If $p>0$ and there exists $x\in X$ such that $p\cdot x\leqslant w$, then $\beta(p,w)$ is compact and convex.
\end{lemma}

\begin{proof}
Convexity is clear. See \cite{dsb_ctsdem} for a proof that $\beta(p,w)$ is compact.
\end{proof}

\bigskip
\noindent
We use $\partial S$ to denote the boundary of a subset $S$ of some metric space.

\begin{lemma}
The boundary of $\beta(p,w)$ is compact.
\end{lemma}

\begin{proof}
If $X$ is colocated, then $\rho(x,\partial X)=\max\{\rho(x,X),\rho(x,-X)\}$ and hence the boundary of $X$ is located. Therefore it suffices to show that $-\beta(p,w)$ is located. This is similar to the proof of Lemma \ref{dsb}.
\end{proof}

\bigskip
\noindent
It now follows from Theorem \ref{t1} that the function $F$, the consumers \emph{demand function}, that maps $(p,w)$, where $p$ is a price vector and $w$ an initial endowment, to the unique maximal element of $\beta(p,w)$, is well defined. By logical considerations we have that any function which can be proven to exist within Bishop's constructive mathematics alone is classically continuous, so the consumers demand function is continuous in the classical setting.

\bigskip
\noindent
We seek conditions under which $F$ is constructively continuous. A function on a locally compact space is said to be \emph{Bishop continuous} if it is pointwise continuous, and is further uniformly continuous on every compact space. Since the uniform continuity theorem---every continuous function on a compact space is uniformly continuous---is not provable in Bishop's constructive mathematics, this is the natural notion of continuity for us to consider. We study the continuity of $F$ by looking at the map $\Gamma$, on the set $T$ of all inhabited $\beta(p,w)$, taking $\beta(p,w)$ to $F(p,w)$. We give $T$ the Hausdorff metric: for located subsets $A,B$ of a metric space $Y$
 $$
  \rho_H(A,B)=\max\left\{\sup\{\rho(a,B):a\in A\},\sup\{\rho(b,A):b\in B\}\right\}.
 $$
\noindent
Our next lemma shows how studying $\Gamma$ allows us to show the continuity of $F$.

\begin{lemma}
\label{lGamma}
 If $\Gamma$ is continuous, then $F$ is continuous. If $\Gamma$ is uniformly continuous, then for each $p\in\mathbf{R}^n$, $w\mapsto F(p,w)$ is uniformly continuous, and for each $w\in\mathbf{R}$, $p\mapsto F(p,w)$ is Bishop continuous.
\end{lemma} 

\noindent
In constructive mathematics, the uniform continuity theorem---every pointwise continuous function with compact domain is uniformly continuous---is closely related to the `semi-constructive' fan theorem isolated by Brouwer. In the appendix we introduce Brouwer's fan theorem (\textbf{FT}) and the notion of a weakly (uniformly) continuous predicate, and we give a version of the uniform continuity theorem for these predicates. Our next result says that adopting Brouwer's fan theorem is sufficient to prove the classical result that $F$ is continuous when $\succ$ is continuous and strictly convex. We observe that if $\beta(p,w)$ is inhabited and every component of $p$ is positive, then 
 $$
  \beta(p,w)=\left\{x\in\mathbf{R}^n:\sum_{i=1}^n p_i x_i\leqslant w\right\}
 $$
\noindent
is a diamond, and if the \emph{diameter}
 $$
  \sup\{\rho(x,y):x,y\in\beta(p,w)\}
 $$
\noindent
of $\beta(p,w)$ is positive, then $\beta(p,w)$ has inhabited interior.

\begin{theorem}
\label{DF1}
 Suppose Brouwer's full fan theorem holds. If $\succ$ is continuous and strictly convex, then $F$ is Bishop continuous.
\end{theorem}

\begin{proof}
Since \textbf{FT} implies that every continuous function on a compact space is uniformly continuous, it suffices, by Lemma \ref{lGamma}, to show that $\Gamma$ is continuous. Fix $\varepsilon>0$, and $(p,w)\in\mathbf{R}^{n+1}$ such that $\beta(p,w)$ is inhabited; we write $S=\beta(p,w)$ and $\xi=F(p,w)$. 

\bigskip
\noindent
Either $\rho(\xi,\partial S)>0$ or $\rho(\xi,\partial S)<\varepsilon/2$. In the first case $\xi$ is maximal on the entire set of consumer bundles, so it suffices to set $\delta=\varepsilon$. In the second case, let $\varphi$ be the natural bijection of $[0,1]^n$ with $T\equiv\partial\beta(p,w)\setminus B_{\rho_1}(x,\varepsilon/2)$; without loss of generality, $\varphi$ is nonexpansive. We define a predicate on $[0,1]$ by
 $$
  P(x,\alpha,\delta)\ \ \Leftrightarrow\ \ \forall_{y\in B(\varphi(x),\delta)}\xi\succ y.
 $$
\noindent
Then $P$ is a weakly continuous predicate: condition (i) follows from Corollary \ref{c1} and the lower pointwise continuity of $\succ$; condition (ii) follows from elementary geometry, given that $\varphi$ is nonexpansive. By Theorem \ref{UCTPred}, $P$ is weakly uniformly continuous and hence there exists $\delta>0$ such that every $y\in B(x,\delta)$ is strictly less preferable than $\xi$ for all $x\in T$. If $\rho(x,S)<\min\{\delta,\varepsilon\}/2$, then $\rho(x,T)<\delta$, $x\in S$, or $x\in B(\xi,\varepsilon)$. In the first two cases $\xi\succ x$; it follows that $F(p^\prime,w^\prime)\in B(\xi,\varepsilon)$ whenever $\rho_H(\beta(p,w),\beta(p^\prime,w^\prime)<\min\{\delta,\varepsilon\}/2$.
\end{proof}

\bigskip
\noindent
It may seem a little odd that we choose to work in Bishop's constructive mathematics because we are interested in producing results with computational meaning, but that we then add an extra principle \textbf{FT} to our framework. In particular, the inconsistency of Brouwer's fan theorem with recursive analysis \cite{BR} may cause some consternation. The constructive nature of the fan theorem can be intuitively justified as follows: in order to assert that $B$ is a bar we must have a proof that $B$ is a bar, and a proof is a finite object; therefore an examination of the finite information used in the proof that $B$ is a bar should reveal the uniform bound that the fan theorem gives us. Although this argument does not hold up in Bishop's constructive mathematics, if your objects are presented in the right way (and indeed a very nature way from a computational point of view), then the fan theorem can be proved \cite{NP,PT}.

\bigskip
\noindent
We pause here to give a consequence of Theorem \ref{DF1}. Consider a system with $N$ commodities, $n$ producers, and $m$ consumers. To each producer we associate a production set $Y_i\subset\mathbf{R}^N$; and to each consumer a consumption set $X_i\subset\mathbf{R}^N$ endowed with a preference relation $\succ_i$. Further we assume that each consumer has no initial endowment. A \emph{competitive equilibrium} of an economy consists of a price vector $\mathbf{p}\in\mathbf{R}^N$, points $\mathbf{\xi}_1,\ldots,\mathbf{\xi}_i\in\mathbf{R}^N$, and a vector $\mathbf{\eta}$ in the \emph{aggregate production set}
 $$
  Y=Y_1+\cdots+Y_n,
 $$
\noindent
satisfying
 \begin{itemize}
  \item[\textbf{E1}] $\mathbf{\xi}_i\in D_i(\mathbf{p})$ for each $1\leqslant i\leqslant m$.
  \item[\textbf{E2}] $\mathbf{p}\cdot\mathbf{y}\leqslant\mathbf{p}\cdot\mathbf{\eta}=0$ for all $\textbf{y}\in Y$.
  \item[\textbf{E3}] $\sum_{i=1}^m\mathbf{\xi}_i=\mathbf{\eta}$.
 \end{itemize}
\noindent
An economy is said to have \emph{approximate competitive equilibria} if for all $\varepsilon>0$ there exist a price vector $\mathbf{p}\in\mathbf{R}^N$, points $\mathbf{\xi}_1,\ldots,\mathbf{\xi}_i\in\mathbf{R}^N$, and a vector $\mathbf{\eta}$ satisfying \textbf{E1},\textbf{E3}, and 
 \begin{itemize}
  \item[\textbf{AE}] $\mathbf{p}\cdot\mathbf{\eta}>-\varepsilon$.
 \end{itemize}

\bigskip 
\noindent
The work in \cite{HM} together with Theorem \ref{DF1} gives the next result, which is an approximate version of McKenzie's theorem on the existence of competitive equilibria \cite{M3}. 

\begin{theorem}
 \label{cMcKenzie}
 Assume that Brouwer's fan theorem holds. Suppose that 
  \begin{itemize}
   \item[(i)] each $X_i$ is compact and convex;
   \item[(ii)] each $\succ_i$ is continuous and strictly convex;
   \item[(iii)] $\left(X_i\cap Y\right)^\circ$ is inhabited for each $i$;
   \item[(iv)] $Y$ is a located closed convex cone;
   \item[(v)] $Y\cap\left\{\left(x_1,\ldots,x_N\right):x_i\geqslant0\mbox{ for each }i\right\}=\{0\}$; and
   \item[(vi)] for each $\mathbf{p}\in\mathbf{R}^N$ and each $i$, if $\sum_{i=1}^m F_i(\mathbf{p})\in Y$, then there exists $\textbf{x}_i\in X_i$ such that $\mathbf{x}_i\succ_i F_i(\mathbf{p})$.
  \end{itemize}
 \noindent
 Then there are approximate competitive equilibria.
\end{theorem}

\subsection*{Uniformly rotund preference relations}

\noindent
In order to prove Theorem \ref{DF1} we effectively strengthened our theory, and therefore weakened our notion of computable. The other natural approach toward proving the existence of a Bishop continuous demand function is to strengthen the conditions on $\succ$. We follow the lead of Bridges in \cite{dsb_ctsdem} and focus on uniformly rotund preference relations.

\bigskip
\noindent
Hereafter, we extend the domain of $\Gamma$ to all inhabited, compact, convex subsets of $X$. Theorem \ref{t1} still ensures that $\Gamma$ is well defined.

\begin{theorem}
\label{lGammacts}
If $\succ$ is a uniformly rotund preference relation, then $\Gamma$ is uniformly continuous.
\end{theorem}

\begin{proof}
Let $S,S^\prime$ be compact, convex subsets of $X$ and let $\xi,\xi^\prime$ be their $\succ$-maximal points.
Fix $\varepsilon>0$ and let $\delta^\prime>0$ be such that if $\Vert x-x^\prime\Vert\geqslant\varepsilon$ ($x,x^\prime\in X$), then for each $z\in B(0,\delta^\prime)$ either $\frac{1}{2}\left(x+x^\prime\right)+z\succ x$ or $\frac{1}{2}\left(x+x^\prime\right)+z\succ x^\prime$, and set $\delta=\min\{\varepsilon,\delta^\prime\}/2$. 

\bigskip
\noindent
If $\rho_H(S,S^\prime)<\delta$, then $\Vert\xi-\xi\Vert$: Let $S,S^\prime$ be such that $\rho_H(S,S^\prime)<\delta$ and suppose that $\Vert\xi-\xi^\prime\Vert>\varepsilon$. Since $S,S^\prime$ are convex
 \[
  S\cap B((\xi+\xi^\prime)/2,\delta)\mbox{ and }S^\prime\cap B((\xi+\xi^\prime)/2,\delta)
 \]
\noindent
are both inhabited; let $z$ be an element of the former set and let $z^\prime$ be an element of the latter. By the maximality of $\xi\in S$ and our choice of $\delta$, $z\succ\xi^\prime$; similarly, $z^\prime\succ\xi$. Therefore
 \[
  \xi\succcurlyeq z\succ\xi^\prime\succcurlyeq z^\prime\succ\xi,
 \]
\noindent
which is absurd. Hence $\Vert\xi-\xi^\prime\Vert\leqslant\varepsilon$.
\end{proof}

\bigskip
\noindent
As a corollary we have the following improvement on the main result of \cite{dsb_ctsdem}. 

\begin{corollary}
\label{DF2}
Let $\succ$ be a uniformly rotund preference relation on a compact,
uniformly rotund subset $X$ of $\mathbf{R}^n$, and let $S$ be a subset of $\mathbf{R}^n\times \mathbf{R}$ such that $\beta(p,w)$ is inhabited for each $(p,w)\in S$. Then for each $p\in\mathbf{R}^n$, the function $w\mapsto F(p,w)$ is uniformly continuous, and for each $w\in\mathbf{R}$, the function $p\mapsto F(p,w)$ is Bishop continuous. In particular, $F$ is Bishop Continuous.
\end{corollary}

\begin{proof}
The result follows directly from Lemma \ref{lGamma} and Theorem \ref{lGammacts}.
\end{proof}

\bigskip
\noindent
Not surprisingly, a less uniform version of rotundness is enough to give us the pointwise continuity of $\Gamma$. A subset $X$ of $\mathbf{R}^n$ is \emph{rotund} if for each $x\in X$ and each $\varepsilon>0$ there exists $\delta>0$ such that for all $x^\prime\in X$, if $\Vert x-x^\prime\Vert\geqslant\varepsilon$, then 
 $$
  \left\{\frac{1}{2}\left(x+x^\prime\right)+z:z\in B(0,\delta)\right\}\subset X.
 $$
\noindent
A preference relation $\succ$ is \emph{rotund} if $X$ is rotund and for each $x\in X,\varepsilon>0$ there exists $\delta>0$ such that if $\Vert x-x^\prime\Vert\geqslant\varepsilon$ ($x^\prime\in X$), then for each $z\in B(0,\delta)$ either $\frac{1}{2}\left(x+x^\prime\right)+z\succ x$ or $\frac{1}{2}\left(x+x^\prime\right)+z\succ x^\prime$.

\begin{theorem}
\label{lGammaptcts}
If $\succ$ is a rotund preference relation, then $\Gamma$ is continuous.
\end{theorem}

\begin{proof}
The proof is, of course very similar to the proof of Theorem \ref{lGammacts}. Let $S$ be a compact, convex subset of $X$ and let $\xi$ be the unique maximal element of $S$. Fix $\varepsilon>0$. Pick $\delta>0$ such that if $\Vert \xi-x\Vert\geqslant\varepsilon$ ($x\in X$), then for each $z\in B(0,\delta)$ either $\frac{1}{2}\left(\xi+x\right)+z\succ \xi$ or $\frac{1}{2}\left(\xi+x\right)+z\succ x^\prime$. If $S^\prime$ is a compact, convex subset of $X$, with maxima $\xi^\prime$, such that $\rho_H(S,S^\prime)<\delta$, then the assumption that $\Vert\xi-\xi^\prime\Vert>\varepsilon$ leads to a contradiction as in the proof of Theorem \ref{lGammacts}.
\end{proof}

\bigskip
\noindent
By the next result, Theorem \ref{lGammacts} can be used to improve on Theorem \ref{DF1}.  

\begin{proposition}
\label{p1}
Assume Brouwer's fan theorem. If $\succ$ is continuous and strictly convex, then $\succ$ is uniformly rotund.
\end{proposition}

\begin{proof}
 Without loss of generality,
  \[
   C=\{(x,y)\in X^2:\Vert x-y\Vert\geqslant\varepsilon\}
  \]
 \noindent
 is compact; moreover
  \[
   P((x,y),\varepsilon,\delta)\equiv \Vert x-y\Vert<\varepsilon\vee\forall_{z\in B}((x+y)/2,\delta)(z\succ x\vee z\succ y)
  \]
 \noindent
 defines a continuous predicate on $C$. Hence $P$ is uniformly continuous by Theorem \ref{UCTPred}, but the uniformity of $P$ says precisely that $\succ$ is uniformly rotund.
\end{proof}

\begin{corollary}
 Suppose Brouwer's full fan theorem holds. If $\succ$ is continuous and strictly convex, then $\Gamma$ is uniformly continuous.
\end{corollary}

\vspace{.5cm}

\subsection*{Appendix: The fan theorem and continuous predicates}

\noindent
Let $2^\mathbf{N}$ denote the space of binary sequences, Cantor's space, and let $2^*$ be the set of finite binary sequences. A subset $S$ of $2^*$ is \emph{decidable} if for each $a\in 2^*$ either $a\in S$ or $a\notin S$. For two elements $u=(u_1,\ldots,u_m),v=(v_1,\ldots,v_n)\in 2^*$ we denote by $u\frown v$ the \emph{concatenation}
 $$
  (u_1,\ldots,u_m,v_1,\ldots,v_n)
 $$
\noindent
of $u$ and $v$. For each $\alpha\in 2^\mathbf{N}$ and each $N\in\mathbf{N}$ we denote by $\overline{\alpha}(N)$ the finite binary sequence consisting of the first $N$ terms of $\alpha$. A set $B$ of finite binary sequences is called a \emph{bar} if for each $\alpha\in 2^\mathbf{N}$ there exists $N\in\mathbf{N}$ such that $\overline{\alpha}(N)\in B$. A bar $B$ is said to be \emph{uniform} if there exists $N\in \mathbf{N}$ such that for each $\alpha\in2^\mathbf{N}$ there is $n\leqslant N$ with $\alpha(n)\in B$. The strongest form of Brouwer's fan theorem is:
 \begin{quote}
  \textbf{FT}: Every bar is uniform.
 \end{quote}
\noindent
Brouwer introduced the fan theorem as a constructive principle and gave a philosophical justification for its use; it is no longer considered a valid principle of core constructive mathematics, but is still used freely by some schools (see \cite{BR}).

\bigskip
\noindent
A predicate $P$ on $S\times \mathbf{R}^+\times \mathbf{R}^+$ is said to be a \emph{continuous predicate on} $S$ if
 \begin{itemize}
 \item[(i)] for each $x\in S$ and each $\varepsilon>0$, there exists $\delta>0$ such that $P(x,\varepsilon,\delta)$;
 \item[(ii)] if $P(x,\varepsilon,\delta)$ and $|x-y|<\delta^\prime<\delta$, then $P(y,\varepsilon,\delta-\delta^\prime)$.
\end{itemize}
\noindent
If in addition, for each $\varepsilon>0$, there exists $\delta>0$ such that $P(x,\varepsilon,\delta)$ for all $x\in S$, then $P$ is a \emph{uniformly continuous predicate on} $S$.

\begin{theorem}
\label{UCT}
The statement
 \begin{quote}
  Every continuous predicate on $[0,1]$ is uniformly continuous. 
 \end{quote}
\noindent
is equivalent to the full fan theorem.
\end{theorem}

\begin{proof}
 Let $P$ be a continuous predicate on $[0,1]$ and fix $\varepsilon>0$. Define a uniformly continuous function $f$ from $2^{\mathbf{N}}$ onto $[0,1]$ by
 $$
  f(\mathbf{\alpha})=\sum_{n=1}^\infty\left(\frac{2}{3}\right)^{n-1}\left(\frac{(-1)^{a_n}+1}{2}\right),
 $$
\noindent
where $\mathbf{\alpha}=\left(a_n\right)_{n\geqslant1}$, and let
 $$
  B=\left\{a\in2^*:\forall_{x\in\left(f(a\frown\mathbf{0}),f(a\frown \mathbf{i}_1)\right)}P\left(x,\varepsilon,(2/3)^{|a|}\right)\right\},
 $$
\noindent
where $\frown$ denotes concatenation, $\mathbf{0}=(0,\ldots)$, and $\mathbf{i}_1=(1,0,\ldots)$. We show that $B$ is a bar. Let $\mathbf{\alpha}\in2^\mathbf{N}$, and, using (i), pick $\delta>0$ such that $P(f(\mathbf{\alpha}),\varepsilon,\delta)$. Pick $n$ such that $(2/3)^{n-1}<2\delta$. Then 
 $$
  \left(f(\overline{\mathbf{\alpha}}(n)\frown\mathbf{0}),f(\overline{\mathbf{\alpha}}(n)\frown \mathbf{i_1})\right)_{(2/3)^n}\subset (f(\mathbf{\alpha})-\delta,f(\mathbf{\alpha}+\delta)).
 $$
\noindent
It follows from condition (ii) that $\mathbf{\alpha}(n)\in B$; whence $B$ is a bar. 

\bigskip
\noindent
By Brouwer's fan theorem, there exists $N>0$ such that for all $\mathbf{\alpha}\in2^\mathbf{N}$ there is $n<N$ with $\mathbf{\alpha}(n)\in B$. Then, by condition (ii), $P\left(x,\varepsilon,(2/3)^N\right)$ for all $x\in[0,1]$.

\bigskip
\noindent
Conversely, let $B$ be a bar that is closed under extension and define a predicate $P$ by 
 \[
  P(x,\varepsilon,\delta)\equiv\forall_x\left(f(x)=\alpha\rightarrow\exists_{N>0}(2^{-N}>\delta\wedge\overline{\alpha}(N)\in B)\right).
 \]
\noindent
It is easy to show that $P$ is a pointwise continuous predicate. Hence $P$ is uniformly continuous; in particular, there exists $\delta>0$ such that $P(x,1,\delta)$ holds for all $x\in[0,1]$. Pick $N>0$ such that $2^{-N}<\delta$. Then for all $\alpha\in2^\mathbf{N}$, $\overline{\alpha}(N)\in B$.
\end{proof}

\bigskip
\noindent
Here is the result we need for the proof of Theorem \ref{DF1}.

\begin{theorem}
\label{UCTPred}
  Assume the fan theorem and let $P$ be a pointwise continuous predicate on $[0,1]^n$. Then $P$ is a uniformly continuous predicate on $[0,1]^n$.
\end{theorem}

\begin{proof}
 We proceed by induction on $n$. The case $n=1$ is just one direction of Theorem \ref{UCT}. Suppose that the result holds for predicates on $[0,1]^{n-1}$, and let $P$ be a predicate on $[0,1]^n$. For each $x$ in $[0,1]$ let $P_x$ be the predicate on $[0,1]$ given by
\[
 P_x(z,\varepsilon,\delta)\Leftrightarrow P((z,x),\varepsilon,\delta).
\]
\noindent
Then, since $P$ is continuous, $P_x$ is a continuous predicate for each $x\in[0,1]$. It follows from our induction hypothesis that each $P_x$ is uniformly continuous. Define a predicate $P^\prime$ on $[0,1]$ by 
\[
 P^\prime(x,\varepsilon,\delta)\ \ \Leftrightarrow \forall_{y\in[0,1]^{n-1}}P_x(y,\varepsilon,\delta).
\]
\noindent
It is easily shown that $P^\prime$ is a continuous predicate and that $P^\prime(x,\varepsilon,\delta)$ holds for all $x\in[0,1]$ if and only if $P(x,\varepsilon,\delta)$ holds for all $x\in[0,1]^n$. By Lemma \ref{UCT}, $P^\prime$ is uniformly continuous; whence $P$ is uniformly continuous.
\end{proof}

\end{document}